\documentclass[11pt,reqno,tbtags,a4paper]{amsart}
\usepackage{amssymb}
\usepackage{url}
\usepackage[square,numbers]{natbib}
\bibpunct[, ]{[}{]}{;}{n}{,}{,}

\title
{Moments of balanced P\'olya urns}

\date{20 May, 2025}

\author{Svante Janson}
\thanks{Partly supported by the Knut and Alice Wallenberg Foundation}
\address{Department of Mathematics, Uppsala University, PO Box 480,
SE-751~06 Uppsala, Sweden}
\email{svante.janson@math.uu.se}
\urladdr{http://www.math.uu.se/svante-janson}


\numberwithin{equation}{section}

\renewcommand\le{\leqslant}
\renewcommand\ge{\geqslant}

\allowdisplaybreaks

\theoremstyle{plain}
\newtheorem{theorem}{Theorem}[section]
\newtheorem{lemma}[theorem]{Lemma}

\theoremstyle{definition}

\newtheorem{problem}[theorem]{Problem}
\newtheorem{remark}[theorem]{Remark}

\newtheorem*{ack}{Acknowledgement}

\theoremstyle{remark}

\newenvironment{romenumerate}[1][-10pt]{
\addtolength{\leftmargini}{#1}\begin{enumerate}
 }{\end{enumerate}}

\newenvironment{PUenumerate}[1][0pt]{
\begin{enumerate}
 }{\end{enumerate}}

\newcounter{oldenumi}
\newenvironment{PUenumerateq}
{\setcounter{oldenumi}{\value{enumi}}
\begin{PUenumerate} \setcounter{enumi}{\value{oldenumi}}}
{\end{PUenumerate}}

\newcounter{thmenumerate}

\newcounter{xenumerate}   

\newcommand{\refT}[1]{Theorem~\ref{#1}}

\newcommand{\refL}[1]{Lemma~\ref{#1}}

\newcommand{\refS}[1]{Section~\ref{#1}}

\newcommand{\refSS}[1]{Section~\ref{#1}}

\newcommand{\refApp}[1]{Appendix~\ref{#1}}

\begingroup
  \divide\count255 by 60
  \multiply\count255 by -60
  \advance\count255 by \time
  \ifnum \count255 < 10 \xdef\klockan{\the\count1.0\the\count255}
  \else\xdef\klockan{\the\count1.\the\count255}\fi
\endgroup

\newcommand{\sumi}{\sum_{i=1}^\infty}

\newcommand{\sumin}{\sum_{i=1}^n}

\newcommand\set[1]{\ensuremath{\{#1\}}}

\newcommand\lrset[1]{\ensuremath{\left\{#1\right\}}}
\newcommand\xpar[1]{(#1)}
\newcommand\bigpar[1]{\bigl(#1\bigr)}
\newcommand\Bigpar[1]{\Bigl(#1\Bigr)}
\newcommand\biggpar[1]{\biggl(#1\biggr)}
\newcommand\lrpar[1]{\left(#1\right)}

\newcommand\abs[1]{|#1|}
\newcommand\bigabs[1]{\bigl|#1\bigr|}
\newcommand\Bigabs[1]{\Bigl|#1\Bigr|}
\newcommand\biggabs[1]{\biggl|#1\biggr|}

\def\rompar(#1){\textup(#1\textup)}    

\newcommand\parfrac[2]{\lrpar{\frac{#1}{#2}}}

\def\xexp(#1){e^{#1}}

\newcommand\ntoo{\ensuremath{{n\to\infty}}}

\newcommand\norm[1]{\|#1\|}
\newcommand\bignorm[1]{\bigl\|#1\bigr\|}
\newcommand\Bignorm[1]{\Bigl\|#1\Bigr\|}
\newcommand\biggnorm[1]{\biggl\|#1\biggr\|}

\newcommand\punkt{.\spacefactor=1000}    
    
\newcommand\ie{i.e\punkt}
\newcommand\eg{e.g\punkt}

\newcommand\cf{cf\punkt}
\newcommand{\as}{a.s\punkt}

\newcommand{\tend}{\longrightarrow}
\newcommand\dto{\overset{\mathrm{d}}{\tend}}

\newcommand\bbR{\mathbb R}
\newcommand\bbC{\mathbb C}

\newcounter{CC}
\newcounter{cc}

\renewcommand\Re{\operatorname{Re}}

\newcommand\E{\operatorname{\mathbb E{}}}
\renewcommand\P{\operatorname{\mathbb P{}}}

\newcommand\Cov{\operatorname{Cov}}

\newcommand\gd{\delta}
\newcommand\gD{\Delta}

\newcommand\gam{\gamma}

\newcommand\gl{\lambda}

\newcommand\gs{\sigma}
\newcommand\gS{\Sigma}

\newcommand\eps{\varepsilon}

\renewcommand\phi{\xxx}  

\newcommand\cE{\mathcal E}
\newcommand\cF{\mathcal F}

\newcommand\indic[1]{\boldsymbol1_{#1}} 

\newcommand\qw{^{-1}}

\newcommand\qq{^{1/2}}

\newcommand\intoi{\int_0^1}

\newcommand\ooo{[0,\infty)}

\newcommand\dd{\,\mathrm{d}}

\newcommand{\ui}{uniformly integrable}
\newcommand\rv{random variable}
\newcommand\lhs{left-hand side}
\newcommand\rhs{right-hand side}

\newcommand\sumiq{\sum_{i=1}^q}
\newcommand\sumjq{\sum_{j=1}^q}

\newcommand\gsa{\gs(A)}

\newcommand\el{E_\lambda}
\newcommand\nl{N_\lambda}
\newcommand\pl{P_\lambda}

\newcommand\nul{\nu_\gl}

\newcommand\xoo{_1^\infty}
\newcommand\ixoo{_{i=1}^\infty}
\newcommand\mgds{martingale difference sequence}

\newcommand{\Holder}{H\"older}
\newcommand{\Polya}{P\'olya}

\begin{document}

\begin{abstract} 
We give  bounds for (central) moments for balanced P\'olya urns
under very general conditions. In some cases, these bounds imply that moment
convergence holds in earlier known results on asymptotic distribution.
The results overlap with previously known results, but are here given more
generally and with a simpler proof.
\end{abstract}

\maketitle

\section{Introduction}\label{S:intro}

A (generalized) \Polya{} urn contains balls of different colours.
A ball is drawn at random from the urn, and is replaced by a (possibly
random)
set of balls that depends on the colour of the drawn balls.
This is repeated \emph{ad infinitum}, and we study the
asymptotic composition of the urn.
For details, and the assumptions used in the present paper, see
\refSS{SSpolya}.

Asymptotic results for \Polya{} urns, including asymptotic distributions
(after suitable normalization), 
have been proved by many authors under various conditions
and in varying generality,
beginning with the
pioneering papers by
\citet{Markov1917} and \citet{EggPol}; 
for the history of \Polya{} urns, see \eg{} \citet{Mahmoud}; see also the
references in \cite{SJ154} and \cite{SJ380}.

It is well-known
that the asymptotic behaviour of a \Polya{} urn depends on the eigenvalues of
the \emph{intensity matrix} of the urn defined in \eqref{A} below, and in
particular on the two largest (in real part) eigenvalues $\gl_1$ and
$\gl_2$.
In particular, 
if the urn satisfies some irreducibility condition
(and some technical conditions)
there is a dichotomy
(or trichotomy if the critical case 
$\Re\gl_2=\frac12\gl_1$ is considered separately):
\begin{romenumerate}
\item 
If $\Re\gl_2\le\frac12\gl_1$ (a  \emph{small urn}), 
then
the number of balls of a given colour is asymptotically normal.
\item 
If $\Re\gl_2>\frac12\gl_1$ (a \emph{large urn}), 
then this is not true: 
then there are
limits in distribution, but the limiting distributions
have no simple description and are (typically, at least) not normal;
furthermore, there may be 
oscillations so that suitable subsequences converge in
distribution but not the full sequence.
\end{romenumerate}
See for example \cite[Theorems 3.22--3.24]{SJ154} for general results of
this type.
\citet{P} showed more precise results for large urns 
(not necessarily irreducible) that are 
\emph{balanced}, see \refSS{SSpolya}.
As another example, for urns that are triangular (and thus not
irreducible), the asymptotic distribution still depends in
important ways on the eigenvalues of $A$, but in a more complicated way,
see \cite{SJ380}.
(Note that
we see again a dichotomy between small and large urns
for the triangular urns 
in \cite[Theorem 1.3(i)--(iii)]{SJ169}
but not in
\cite[Theorem 1.3(iv)--(v)]{SJ169}.)

In the present paper we give no results on asymptotic distributions;
instead we consider as a complement the question whether 
moment convergence holds
in
such results on asymptotic distributions.
This too is not a new subject. 
For example, for balanced small urns with 2 colours,
\citet{Bernstein1,Bernstein2}  showed
asymptotic normality in the small urn case and gave results on mean and
variance; 
\citet{BagchiPal} (independently, but 45 years later)
gave another proof of asymptotic normality
using the method of moments and thus
proving moment convergence as part of the proof.
Other examples are 
\citet{BaiHu,BaiHu2005}, who consider a
version of \Polya{} urns allowing  time-dependent replacements
and show  for small urns both
asymptotic normality and (implicitly)
asymptotic results for mean and variance.
The results by \citet{P} for balanced large urns include convergence
in $L^p$ for any $p$, and thus moment convergence.
\citet{SJ310} proved moment convergence
for irreducible small urns; the method there combined the known result on
asymptotic normality  in this case, and moment estimates by the method of
\cite{P} leading to uniform integrability. 
\citet{SJ307} gave asymptotics for mean and variance of balanced small urns
(consistent with known results on asymptotic normality in the irrdeucible case,
but more general).

In the present paper (as in many of the references above) we consider only 
\emph{balanced} urns, but otherwise our conditions are very general.
(For example, we allow random replacements.)
The main purpose of the  paper is to prove 
general \emph{bounds} for moments
of an arbitrary fixed order (see \refS{Smain}); 
these bounds imply uniform integrability and
can thus in many cases be combined with known results on convergence in
distribution to show that moment convergence hold in the latter results
(see for example \refT{T3}).
Our results are to a large extent not new, but as far as we know stated in
greater generality than earlier; moreover, the method of proof seems simpler
that the ones used earlier.
For earlier related results, as said above, \citet{P} showed moment
convergence for balanced large urns; 
\citet{SJ310} showed moment bounds (and
thus convergence) for balanced small irreducible urns; 
\citet{SJ307} treated first and second moments for balanced 
(and often small) urns -- the method used here is a further development and
simplification of the method in \cite{SJ307}.

We do not make any assumptions on the structure of the urn in this paper
(except being balanced). However, the results are particularly adapted to
the irreducible case, where the bounds that we obtain match the
normalizations in the known results on asymptotic normality, thus showing
that these results hold with moment convergence.
The results below apply also to, for example, balanced 
triangular urns, but in that
case the results are in general less sharp and there is often a
gap between the general bounds proved here and the moment convergence proved
for such urns in \cite[Theorem 12.5]{SJ380}.
For example, for a balanced triangular urn with 2 colours,
see \cite[Example 14.4 Case 4]{SJ380}
and \cite[Theorem 1.3(v)]{SJ169}, 
the bounds below are sharp if the urn is large
but not if it is small.

Our results for balanced urns are very general, but that leaves 
one obvious open problem:

\begin{problem}
In the present paper,
we consider only balanced urns.
We leave it as a challenging open problem to prove (or disprove?)
similar results for unbalanced urns.
\end{problem}

In fact, it seems that only essentially trivial examples are known  
of an unbalanced urn where 
(after suitable normalization) 
convergence in distribution holds with all moments.
In the negative direction,
\cite[Example 14.2]{SJ380} gives an example (an unbalanced diagonal urn)
with convergence in distribution
to a limit with infinite mean; hence moment
convergence does not hold. However, this counterexample seems rather special
and is maybe not typical.

\refS{Sprel} gives definitions and some other preliminaries.
\refS{Smain} contains the statements of the main results, which are proved
in Sections \ref{Spf1}--\ref{Smatrix}.
\refApp{Aui} gives some further, more technical, results,
and \refApp{AP} a  proof of a simple lemma for which we do not know a reference.

\begin{ack}
  This paper  owes much to an anonymous referee of my earlier paper
  \cite{SJ307} dealing with first and second moments.
The referee suggested an alternative method of proof of  results there,
using \refL{LLui} below, and 
also said that this could be extended to prove convergence of higher moments.
I am very grateful for that suggestion, which 
(although not used in \cite{SJ307}) 
I have developed and further simplified, leading to 
the present paper.
\end{ack}

\section{Preliminaries and notation}\label{Sprel}

\subsection{Definition and assumptions}\label{SSpolya}

A (generalized) \Polya{} urn process is defined as a discrete-time Markov
process of the following type:
\begin{PUenumerate}[-15pt]
  
\item \label{Po1}
The state of the urn at time $n$ is given by the vector
$X_n=(X_{n1},\dots,X_{nq})\allowbreak\in \ooo^q$ for some given integer $q\ge2$,
where $X_{ni} $ is interpreted as the number of balls of colour $i$ in the urn;
thus balls can have the $q$ colours (types) $1,\dots,q$.
The urn starts with a given vector $X_0$.

\item\label{Po2} 
Each colour $i$ has a given \emph{activity} (or weight) $a_i\ge0$,
and a (generally random) 
\emph{replacement vector} $\xi_i=(\xi_{i1},\dots, \xi_{iq})$.
At each time $ n+1\ge 1 $, 
the urn is updated by drawing
one ball at random from the urn, with the probability of any ball
proportional to its activity. Thus, the drawn ball has colour $i$  
with probability 
\begin{equation}
\label{urn}
 \frac{a_iX_{ni}}{\sum_{j}a_jX_{nj}}. 
\end{equation}
If the drawn ball has colour $i$, it is replaced
together with $\Delta X_{nj}$ balls of colour $j$, $j=1,\dots,q$,
where the random vector 
$ \Delta X_{n}=(\Delta X_{n1},\dots, \Delta X_{nq}) $
has the same distribution as $ \xi_i$ and is
independent of everything else that has happened so far.  
Thus, the urn is updated to $X_{n+1}=X_n+\gD X_n$.
\end{PUenumerate}

In many applications, the numbers $X_{nj}$ and $\xi_{ij}$ are integers, but
that is not necessary;
as is well-known, the \Polya{} urn process
is well-defined also for \emph{real} $X_{ni}$ and $\xi_{ij}$, 
with probabilities for the different replacements still given by \eqref{urn};
the ``number of balls'' $X_{ni}$ may thus be any nonnegative real
number. (This can be interpreted as the amount (mass) of colour $i$ in the
urn, rather than the number of discrete balls.) 
The replacements $\xi_{ij}$ are thus in general random real numbers;
we allow them to be negative,
meaning that balls may be subtracted from the urn.
However, we always assume that $X_0$ and the random vectors $\xi_i$ are
such that,
for every $n\ge0$,  \as
\begin{equation}\label{ten}
\text{each }
 X_{ni}\ge0 \quad \text{and}\quad  \sum_i a_i X_{ni}>0, 
\end{equation}
so that
\eqref{urn} really gives meaningful probabilities,
and the process does not stop due to lack of balls to be removed. 
An urn with such initial
conditions and replacement rules is called \emph{tenable}.


Note that we allow some activities $a_i=0$ (as long as \eqref{ten} holds);
this means that balls of colour $i$ never are drawn. (This is useful in some
applications, see \eg{} \cite{SJ154}.)

For simplicity,
we assume throughout this paper also:
\begin{PUenumerateq}
\item \label{Po0}
The initial state $X_0$ is nonrandom. 
\end{PUenumerateq}

\begin{remark}
The results below can be extended to random $X_0$ by
conditioning on $X_0$, but 
we have not checked exactly what conditions are needed, and we
leave this to the reader.
\end{remark}

We will in the present paper (as in \cite{SJ307}) only consider 
balanced \Polya{} urns, defined as follows:
\begin{PUenumerateq}
\item \label{Pobal}
The \Polya{} urn is \emph{balanced} if 
\begin{equation}\label{balanced}
   \sum_j a_j\xi_{ij} = b>0
\end{equation}
(\as)
for some constant $b$ and every $i$. In other words, the added activity
after each draw is fixed (nonrandom and not depending on the colour of the
drawn ball). 
\end{PUenumerateq}

The balance condition \ref{Pobal} together with \ref{Po0} imply that the 
denominator in \eqref{urn} 
(i.e., the total activity in the urn)
is deterministic for each $n$, see \eqref{wn}--\eqref{wn1}.
This is a significant simplification,  assumed in many papers on
\Polya{} urns, and luckily satisfied in many applications.

\subsection{Some notation}
We regard all vectors as column vectors.
We use standard notations for (real or complex)
vectors and matrices (of sizes $q$ and
$q\times q$, respectively),
in particular ${}'$ for transpose;
we also use
$\cdot$ for the bilinear scalar product defined by
$u\cdot v=u'v$ for any vectors $u,v\in\bbC^q$. 
(This is the standard scalar product for $u,v\in\bbR^q$,
but note the abscence of conjugation in general.)
We  denote the standard Euclidean norm for vectors
by  $\abs{\cdot}$, 
and denote
the operator norm 
for matrices by $\norm{\cdot}$.

Let $a:=(a_1,\dots,a_q)'$ be the vector of activities.
Thus, the balance condition \eqref{balanced} can be written $a\cdot\xi_i=b$
a.s.\ for all $i$.

We let $[q]:=\set{1,\dots,q}$, the (finite) set of colours.

For a random variable (or vector) $X$, we denote the usual $L^p$-norm by 
$\norm{X}_p:=(\E|X|^p)^{1/p}$, $p\ge1$.

We let $C$ denote unspecified constants, possibly different at each occurrence.
They will in general depend on the \Polya{} urn, 
i.e., on $q$, $X_0$ and the distribution of $\xi_i$, $i\in[q]$,
but the do not depend on $n$.
We similarly use $C_p$ for constants that also may depend on the power $p\ge2$.

\subsection{The intensity matrix}
The \emph{intensity matrix} of the \Polya{} urn is
the $ q\times q $  matrix
\begin{equation}\label{A}
A:=(a_j\E\xi_{ji})_{i,j=1}^{q}.   
\end{equation}
Note that, for convenience and following \cite{SJ154} and \cite{SJ307}, 
we have defined $A$
so that the element $(A)_{ij}$ is a measure of the intensity of adding balls
of colour $i$ coming from drawn balls of colour $j$; the transpose matrix
$A'$ is often used in other papers.
(We may unfortunately have contributed to notational confusion by this choice in
\cite{SJ154}.)  
It is well-known that
the intensity matrix $ A $ and in particular its eigenvalues and
eigenvectors have a central role for  asymptotical results.
%

\subsection{Eigenvalues and spectral decomposition}
We shall use the Jordan decomposition of the matrix $A$ in the
following form. 
There exists a decomposition of the {complex} space $\bbC^q$ as a
direct sum $\bigoplus_\gl \el$ of generalized eigenspaces $\el$, such
that $A-\gl I$ is a nilpotent operator on $\el$;
here 
$I$ is the identity matrix 
and
$\gl$ ranges over the set $\gs(A)$ of eigenvalues of $A$.
(In the sequel, $\gl$ will always denote an eigenvalue.)
In other words, there exist projections $\pl$, $\gl\in\gs(A)$, 
that commute with $A$ and satisfy
\begin{gather}
\sum_{\gl\in\gsa}\pl=I,
\label{pl}\\
A\pl=\pl A=\gl\pl+\nl,
\label{24b}
\end{gather}
where $\nl=\pl\nl=\nl\pl$ is nilpotent. 
Moreover, $\pl P_\mu=0$ when $\gl\neq\mu$.
We let 
$\nul\ge0$ be the integer such that $\nl^{\nul}\neq0$ but $\nl^{\nul+1}=0$.
(Equivalently, in the Jordan normal form of $A$, the largest Jordan
block with $\gl$ on the diagonal has size $\nul+1$.) 
Hence $\nul=0$ if and only if $\nl=0$, and this happens for all $\gl$
if and only if $A$ is diagonalizable. 

The eigenvalues of $A$ are denoted $\gl_1,\dots,\gl_q$ (repeated according
to their algebraic multiplicities); we assume that they are ordered 
with decreasing real parts: $\Re\gl_1\ge\Re\gl_2\ge\dots$, and furthermore,
when the real parts are equal, in order of decreasing $\nu_j:=\nu_{\gl_j}$. 
In particular, 
if $\gl_1>\Re\gl_2$, then $\nu_j\le\nu_2$ for every eigenvalue $\gl_j$
with $\Re\gl_j=\Re\gl_2$.

Recall that the urn is called \emph{small} if $\Re\gl_2\le\frac12\gl_1$
and \emph{large} if $\Re\gl_2>\frac12\gl_1$. 

In the balanced case, by \eqref{A} and \eqref{balanced},
\begin{equation}\label{bala}
  a'A
=\Bigpar{\sumiq a_i(A)_{ij}}_j
=\Bigpar{\sumiq a_ia_j\E \xi_{ji}}_j
=\bigpar{a_j\E (a\cdot \xi_{j})}_j
=ba',
\end{equation}
\ie, $a'$ is a left eigenvector of $A$ with eigenvalue $b$.
Thus $b\in\gsa$.
We shall assume that, moreover, $b$ is the largest eigenvalue, \ie,
\begin{equation}\label{gl1b}
  \gl_1=b.
\end{equation}
This is a very weak assumption; see \cite[Appendix A]{SJ307}
where it is shown that there are counterexamples but only in a trivial manner.
In particular, the following is shown there.
\begin{lemma}[{\cite[Lemma A.1]{SJ307}}]\label{Lapp}
  If the \Polya{} urn is tenable and balanced, 
and moreover any colour has a nonzero probability of ever appearing in the urn,
then $\Re\gl\le b$ for every
  $\gl\in\gs(A)$, and, furthermore, if\/ $\Re\gl=b$ then $\nul=0$.
We may thus assume $\gl_1=b$.
\end{lemma}

The eigenvalue $b$ may be multiple; a well-known example is the classical
\Polya{} urn for which $A=bI$, and thus all $q$ eigenvalues are equal to $b$.
In most other applications, the eigenvalue $b$ is simple.
This implies that the corresponding left and right eigenspaces are
1-dimensional, and thus the corresponding left and right eigenvectors $u_1$
and $v_1$ are unique up to constant factors.
By \eqref{bala}, $a'$ is a left eigenvector so we may then take $u_1=a$.
Furthermore, we have the following  general result from linear algebra; 
for completeness we give a proof in \refApp{AP} 
since we do not know a good reference.
\begin{lemma}\label{Lsimple}
  Suppose that the eigenvalue $\gl_1=b$ is simple.
Then there is a unique right eigenvector $v_1$ with
\begin{align}
\label{normalized}
u_1\cdot v_1=a\cdot v_1=1.
\end{align}
Furthermore, the projection $P_{\gl_1}$ is given by
\begin{equation}
  \label{pl1}
P_{\gl_1}=v_1u_1'.
\end{equation}
Consequently, 
for any vector $v\in\bbC^q$,
\begin{equation}\label{pl2}
  P_{\gl_1}v=v_1u_1'v=v_1a'v=(a\cdot v)v_1.
\end{equation}
\end{lemma}

\section{Main results}\label{Smain}
We state here our main results, using the notation and assumptions above.
Proofs are given in \refS{Spf1}.
We begin with a general upper bound for the moments.
\begin{theorem}\label{T2}
Assume that the \Polya{} urn is tenable and balanced with $\gl_1=b$.
  Let $p\ge2$ and suppose that
  \begin{align}
    \label{t1p}
\E|\xi_{ij}|^p<\infty 
\qquad\text{for all $i,j\in[q]$}.
  \end{align}
Then,
for every $n\ge2$,
  \begin{align}\label{t2}
\bignorm{X_n-\E X_n}_p
\le
    \begin{cases}
 C_p n\qq,& \Re{\gl_2} <\gl_1/2,\\
 C_p n\qq (\log n)^{\nu_{2}+\frac12},& \Re{\gl_2} =\gl_1/2,\\
 C_p n^{\Re {\gl_2}/\gl_1} (\log n)^{\nu_{2}},& \Re{\gl_2} >\gl_1/2,
    \end{cases}
  \end{align}
for some constant $C_p$ not depending on $n$.
\end{theorem}

As said in the introduction,
in many cases the asymptotic distribution of the urn is known. 
Furthermore, for an irreducible urn, under some technical conditions,
the bound in \eqref{t2} equals (apart from the constant $C_p$) the right
normalizing factor, see for example \cite[Theorems 3.22--24]{SJ154}. 
In particular, for a small irreducible urn (again under some conditions),
we have asymptotic normality as in \eqref{t3a} or \eqref{t3b} below.
The following theorem shows that then also all moments (ordinary and absolute)
converge.

\begin{theorem}\label{T3}
Assume that the \Polya{} urn is tenable and balanced,
and
suppose that \eqref{t1p} holds for every $p\ge1$.
\begin{romenumerate}
  
\item \label{T3a}
Assume that\/ $\Re\gl_2<\gl_1/2$ and that, as \ntoo, 
we have asymptotic normality
\begin{align}\label{t3a}
  \frac{X_n-\E X_n}{\sqrt n}\dto N(0,\gS)
\end{align}
for some  matrix $\gS$. Then \eqref{t3a} holds with convergence of
all moments.
In particular, the covariance matrix $\Cov[X_n]/n\to \gS$.

\item \label{T3b}
Assume that\/ $\Re\gl_2=\gl_1/2$ and that, as \ntoo, 
we have asymptotic normality
\begin{align}\label{t3b}
  \frac{X_n-\E X_n}{\sqrt{n(\log n)^{2\nu_2+1}}}\dto N(0,\gS)
\end{align}
for some matrix $\gS$. Then \eqref{t3b} holds with convergence of all moments.
In particular, the covariance matrix 
$\Cov[X_n]/\bigpar{n(\log n)^{2\nu_2+1}}\to\gS$.
\end{romenumerate}
\end{theorem}

We may also estimate different spectral projections separately and get a
sharper result than \refT{T2}.

\begin{theorem}\label{T1}
Assume that the \Polya{} urn is tenable and balanced.
  Let $p\ge2$ and suppose that \eqref{t1p} holds.
Let $\gl$ be an eigenvalue of $A$. Then,
for every $n\ge2$,
  \begin{align}\label{t1}
\bignorm{P_\gl(X_n-\E X_n)}_p
\le
    \begin{cases}
 C_p n\qq,& \Re\gl <b/2,\\
 C_p n\qq (\log n)^{\nu_\gl+\frac12},& \Re\gl =b/2,\\
 C_p n^{\Re \gl/b} (\log n)^{\nu_\gl},& \Re\gl >b/2,
    \end{cases}
  \end{align}
for some constant $C_p$ not depending on $n$.
\end{theorem}

As in \refT{T3}, one can often combine \refT{T1} with a central limit result
for a specific components $P_\gl(X_n-\E X_n)$ (or for suitable linear
combinations $u\cdot(X_n-\E X_n)$) and obtain moment convergence in the
latter results, cf.\ \cite[Remark 3.25]{SJ154} and 
\cite[Theorem 3]{AthreyaKarlin}; we leave the details to the reader.

For the special case $\gl=\gl_1$, and further assuming that  this eigenvalue is
simple, we have  as a complement the following almost
trivial result.

\begin{theorem}\label{T0}
Assume that the \Polya{} urn is tenable and balanced 
and  that $\gl_1=b$ is a simple eigenvalue. Then
  \begin{align}\label{t0}
P_{\gl_1}(X_n-\E X_n)=0.    
  \end{align}
\end{theorem}

When $\Re\gl\le b/2$, we have under quite general condition
asymptotic normality of $P_\gl(X_n-\E X_n)$, for example as a consequence
of \cite[Theorem 3.15]{SJ154}.
In such cases, 
we obtain from \refT{T1} also moment convergence of $P_\gl(X_n-\E X_n)$, by
the same argument as in the proof of \refT{T3}.

\begin{remark}
  The assumption \eqref{t1p} on finite moments of the replacements is a very
  weak restriction in the balanced case studied here.
For example, 
for a balanced and tenable urn such that
every colour may appear with positive probability 
(which we may assume without loss of generality) and all activities $a_i>0$,
we necessarily have all $\xi_{ij}$ bounded by some constant, and thus
\eqref{t1p} holds, 
see \cite[Remark 2.5]{SJ307}.
Nevertheless, we note for completeness that if we assume only that
\eqref{t1p} holds for a single $p\ge2$, then our proof of \refT{T3} in
\refS{Spf1} yields only moment convergence for 
moments of order strictly less that $p$; however, we show in \refApp{Aui} a
modification of the argument which yields convergence also of moments of
order $p$ in this case.
\end{remark}

\section{Proofs of main results}\label{Spf1}
We prove here the main results in \refS{Smain}, using some 
lemmas postponed to the following sections.
We begin with standard arguments, partly copying from \cite{SJ307}.

Let $I_n$ be the colour of the $n$-th drawn ball, and let
\begin{equation}\label{gDX}
  \gD X_n:=X_{n+1}-X_n
\end{equation}
and
\begin{equation}\label{wn}
  w_n:=a\cdot X_n,
\end{equation}
the total weight (activity) of the urn.
We note first that the assumption that the urn is balanced implies
$a\cdot\gD X_n=b$ a.s., and thus \eqref{gDX}--\eqref{wn} imply that $w_n$ is
deterministic with
\begin{equation}\label{wn1}
  w_n=w_0+nb,
\end{equation}
where the initial weight $w_0=a\cdot X_0$.

Next, let $\cF_n$ be the $\gs$-field generated by $X_1,\dots,X_n$.
Then, by the definition of the urn,
\begin{equation}\label{pin}
  \P\bigpar{I_{n+1}=j\mid \cF_n} = \frac{a_jX_{nj}}{w_n}
\end{equation}
and consequently, recalling \eqref{A},
\begin{equation}\label{mia}
  \begin{split}
\E\bigpar{\gD X_n\mid \cF_n}
&=\sumjq \P\bigpar{I_{n+1}=j\mid \cF_n}\E\xi_j	
=\frac{1}{w_n}\sumjq a_j X_{nj}\E\xi_j	
\\&
=\frac{1}{w_n}\biggpar{\sumjq (A)_{ij} X_{nj}}_i
=\frac{1}{w_n} A X_{n}.
  \end{split}
\end{equation}

Define
\begin{equation}\label{Yn}
  Y_n:=\gD X_{n-1} - \E\bigpar{\gD X_{n-1}\mid \cF_{n-1}},
\qquad n\ge1.
\end{equation}
Then, $Y_n$ is $\cF_n$-measurable and, obviously,
\begin{equation}
  \label{eyn}
\E\bigpar{Y_n\mid\cF_{n-1}}=0.
\end{equation}
In other words, $(Y_n)\xoo$ is a martingale difference sequence.

Furthermore, similarly to \eqref{mia} and using the assumption \eqref{t1p},
\begin{align}\label{sjw}
\E\bigpar{|\gD X_n|^p\mid \cF_n}
&=\sumjq \P\bigpar{I_{n+1}=j\mid \cF_n}\E|\xi_j|^p	
\le C_p
\quad\text{a.s.}
\end{align}
Hence,
$ \E|\gD X_n|^p \le C_p$, or equivalently
\begin{align}\label{Xnp}
  \norm{\gD X_n}_p\le C_p,
\end{align}
which by \eqref{Yn} implies
\begin{align}\label{Ynp}
  \norm{Y_n}_p\le C_p.
\end{align}

By \eqref{gDX}, \eqref{Yn} and \eqref{mia},
\begin{equation}
  X_{n+1}=X_n+Y_{n+1}+w_n\qw A X_n
=\bigpar{I+w_n\qw A}X_n+Y_{n+1}.
\end{equation}
Consequently, by induction, for any $n\ge0$,
\begin{equation}\label{tia}
  X_n=\prod_{k=0}^{n-1}\bigpar{I+w_k\qw A}X_0
  +\sum_{\ell=1}^{n}\prod_{k=\ell}^{n-1}\bigpar{I+w_k\qw A}Y_\ell,
\end{equation}
where (as below) an empty matrix product is interpreted as $I$.

We define the (deterministic) matrix products
\begin{equation}\label{qia}
  F_{i,j}:=\prod_{i\le k<j}\bigpar{I+w_k\qw A},
\qquad 0\le i\le j,
\end{equation}
and write  \eqref{tia} as
\begin{equation}
  \label{kia}
  X_n=F_{0,n} X_0+\sum_{\ell=1}^{n}F_{\ell,n} Y_{\ell}.
\end{equation}

Taking the expectation we find, since $\E Y_\ell=0$ by \eqref{eyn}, 
and the $F_{i,j}$ and $X_0$ are nonrandom,
\begin{equation}
  \label{EX}
\E X_n=F_{0,n} X_0.
\end{equation}
Hence, \eqref{kia} can also be written
\begin{equation}\label{gw}
  X_n-\E X_n =\sum_{\ell=1}^{n}F_{\ell,n} Y_\ell.
\end{equation}

\begin{proof}[Proof of \refT{T1}]
It follows from \eqref{gw} that
\begin{equation}\label{gwp}
P_\gl( X_n-\E X_n) =\sum_{\ell=1}^{n}P_\gl F_{\ell,n} Y_\ell.
\end{equation}
We now use \refL{LL2} below
and
conclude using \eqref{Ynp} that
\begin{equation}\label{hgwp}
\bignorm{P_\gl(X_n-\E X_n)}_p
\le C_p \biggpar{\sum_{\ell=1}^{n}\bignorm{P_\gl F_{\ell,n}}^2}\qq
.\end{equation}
The result follows by \refL{Lsoff}.
\end{proof}

\begin{proof}[Proof of \refT{T0}]
  By \eqref{wn} and \eqref{wn1}, $a\cdot X_n=w_n$ is deterministic, and thus
  \begin{align}
a\cdot(X_n-\E X_n)=a\cdot X_n -E(a\cdot X_n)=0.    
  \end{align}
The result now follows from \eqref{pl2} in \refL{Lsimple}.
\end{proof}

\begin{proof}[Proof of \refT{T2}]
  First, if $\Re\gl_2=\gl_1=b$,  then we simply use Minkowski's inequality,
  which by \eqref{Xnp} yields
  \begin{align}
    \norm{X_n}_p
\le 
    \norm{X_0}_p + \sum_{i=0}^{n-1} \norm{\gD X_i}_p
\le C_p+C_p n
\le C_p n.
  \end{align}
Hence, $\norm{\E X_n}_p\le\norm{X_n}_p\le C_p n$ and 
$\norm{X_n-\E X_n}\le C_p n$, which is \eqref{t2} in the case $\Re\gl_2=\gl_1$.

In the rest of the proof, suppose instead that $\Re\gl_2<\gl_1$. In
particular (since eigenvalues are counted with multiplicities), $\gl_1$ is a
simple eigenvalue.
Thus \refT{T0} applies and shows
$P_{\gl_1}(X_n-\E X_n)=0$. The decomposition \eqref{pl} thus yields
\begin{align}
  X_n-\E X_n = \sum_{\gl\neq\gl_1}P_\gl(X_n-\E X_n),
\end{align}
and Minkowski's inequality yields
\begin{align}
\norm{X_n-\E X_n}_p \le \sum_{\gl\neq\gl_1} \norm{P_\gl(X_n-\E X_n)}_p.
\end{align}
We estimate the terms on the \rhs{} by \refT{T1}; the  contribution
from $\gl=\gl_2$ dominates all others, and we obtain \eqref{t2}.
\end{proof}

\begin{proof}[Proof of \refT{T3}]
  It follows from \refT{T1} that for every $p\ge1$, the $L^p$ norms of
  the \lhs{s} of \eqref{t3a} and \eqref{t3b} are bounded as \ntoo.
As is well-known this implies that for every $p\ge1$, the $p$th powers
$|X_n-\E X_n|^p$ are uniformly integrable, and hence convergence of moments
follows from the assumed convergence in distribution.
\end{proof}

\section{A martingale inequality}\label{Smart}
We used above the following martingale inequality, which is a simple
consequence of Burkholder's inequality for the square function.
Since we do not know a reference where this inequality is stated in the form
below, we give a complete proof.
Recall that a \mgds{} is a sequence $(Y_i)\xoo$ of \rv{s} such that
the sequence $\sumin Y_i$, $n\ge1$, is a martingale with respect to some
sequence of $\gs$-fields $\cF_n$. (The $\gs$-fields $\cF_n$ will be fixed
below.)

\begin{lemma}\label{LL2}
Let $p\ge2$ and let\/ $Y_i$, $i\ge1$, be a \mgds{} 
of random vectors in $\bbC^q$
such that\/
$\sup_i \norm{Y_i}_p<\infty$.
Then, for any sequence of (nonrandom) $q\times q$ matrices $(A_i)\ixoo$
such that\/ 
$\sumi\norm{A_i}^2<\infty$, 
the sum
$\sumi A_i Y_i$ converges a.s.\ and in $L^p$, and
\begin{align}\label{ll2}
\biggnorm{\sumi A_i Y_i}_p \le C_p 
\lrpar{\sumi\norm{A_i}^2}\qq
\sup_i\norm{Y_i}_p.
\end{align}
Here $C_p$ is a constant that depends on $p$ and $q$ only.
\end{lemma}

\begin{proof}
$X_n:=\sumin A_i Y_i$, $n\ge0$, is a martingale, and its square function is
\begin{align}
  S_n(X):=
\biggpar{\sumin|X_i-X_{i-1}|^2}\qq
=\biggpar{\sumin|A_iY_i|^2}\qq.
\end{align}
  Let $B:=\sup_i\norm{Y_i}_p$.
Minkowski's inequality yields, for any $n\ge1$,
\begin{align}
  \norm{S_n(X)}_p^2&
= \bignorm{S_n(X)^2}_{p/2}
= \biggnorm{\sumin |A_iY_i|^2}_{p/2}
\notag\\&
\le\sumin \bignorm{|A_iY_i|^2}_{p/2}
\le \sumin \norm{A_i}^2\norm{|Y_i|^2}_{p/2}
= \sumin \norm{A_i}^2\norm{Y_i}_{p}^2
\notag\\&
\le B^2 \sumin \norm{A_i}^2
\le B^2 \sumi \norm{A_i}^2,
\end{align}
\ie, $  \norm{S_n(X)}_p\le B \bigpar{\sumi \norm{A_i}^2}\qq$.
We combine this with Burkholder's inequality 
\begin{align}\label{em3}
\norm{X_n}_p \le C_p \norm{S(X)_n}_p 
\end{align}
(valid for any martingale and any $p\ge1$, with some constant $C_p$
depending on $p$ only),
see
\cite[Theorem 9]{Burkholder} or e.g.\
\cite[Theorem 10.9.5]{Gut}.
This yields
\begin{align}\label{ll2a}
  \norm{X_n}_p \le C_p \norm{S_n(X)}_p\le C_pB \biggpar{\sumi \norm{A_i}^2}\qq.
\end{align}
Consequently, 
the martingale $(X_n)\xoo$ is $L^p$ bounded, and thus converges \as{}
and in $L^p$ to a limit, which satisfies \eqref{ll2} by \eqref{ll2a}.
\end{proof}

\begin{remark}
We have in the proof above used a vector-valued version of Burkholder's
  inequality \eqref{em3};
this follows immediately
(with a constant $C_p$ depending on $q$)
from the scalar-valued version in the references above applied to each
component. 
In fact (for $p\ge2$),
\eqref{em3} holds with $C_p=p-1$ independent of the dimension $q$,
and, in fact, more generally for Hilbert-space-valued martingales, see
\cite[Theorem 3.3]{Burkholder91}.
\end{remark}

\section{Matrix estimates}\label{Smatrix}
We prove here some matrix estimates used in the proofs above.
In this section, we may as well be general and let
$A$ be any complex $q\times q$ matrix.
Let $P_\gl$, $\gl\in\gs(A)$,
be its spectral projections as in \eqref{pl}--\eqref{24b}, 
and let $\nu_\gl+1$ be the dimension of the largest Jordan block
for the eigenvalue $\gl$.

Furthermore, we assume that $w_0$ and $b>0$ are some given positive numbers.
We then define $w_n:=w_0+nb>0$ as in \eqref{wn1}, 
and define the matrices $F_{i,j}$ by \eqref{qia}.
\begin{lemma}\label{Lsof}
With notations as above, for every eigenvalue $\gl\in\gs(A)$,
\begin{align}\label{lsof}
  \norm{P_\gl F_{i,j}}
\le C \parfrac{j}{i}^{\Re\gl/b}\Bigpar{1+\log \frac ji}^{\nu_\gl},
\qquad 1\le i\le j<\infty
.\end{align}
for some constant $C$ not depending on $i$ and $j$.
\end{lemma}
\begin{proof}
We change basis in $\bbC^q$ so that $A$ is reduced to Jordan normal form.
(This may change the matrix norms, but at most by some multiplicative
constants, which are incorporated in the final $C$ and do not affect the
result.) In this basis, $A$ has one or several Jordan blocks with diagonal
element $\gl$. Multiplying by $P_\gl$ kills all Jordan blocks with
eigenvalue $\neq \gl$, and as a result it suffices to consider a single
Jordan block with eigenvalue $\gl$.
We may thus assume that $A=\gl I+N$ where $N$ is nilpotent, \cf{} \eqref{24b};
more precisely we have $N^{\nu_\gl+1}=0$ by the definition of $\nu_\gl$.
Assume also that $i$ is so large that $w_i\ge 2|\gl|$, say, and thus in
particular 
$w_k+\gl\neq0$ for $k\ge i$.
In this case, by \eqref{qia}, 
\begin{align}\label{sof1a}
   P_\gl F_{i,j}&
=\prod_{i\le k<j}\bigpar{I+w_k\qw(\gl I+N)}
\\&\label{sof1}
= \prod_{i\le k<j}(1+\gl/w_k)
\prod_{i\le k<j}\bigpar{I+(w_k+\gl)\qw N}
.\end{align}
The first product on the \rhs{} of \eqref{sof1} is a complex number which
can be estimated by (since we assume $w_k\ge w_i\ge 2|\gl|$)
\begin{align}\label{sof2}
  \prod_{i\le k<j}(1+\gl/w_k)&
=\exp\biggpar{\sum_{i\le k<j}\log(1+\gl/w_k)}
\notag\\&
=\exp\biggpar{\sum_{i\le k<j}\Bigpar{\frac{\gl}{w_k}
+O\Bigpar{\frac{\gl^2}{w_k^2}}}}
\notag\\&
=\exp\biggpar{\sum_{i\le k<j}\Bigpar{\frac{\gl}{kb}
+O\Bigpar{\frac{1}{k^2}}}}
\notag\\&
=\exp\biggpar{\frac{\gl}{b}\sum_{i\le k<j}\frac{1}{k}
+O(1)}
.\end{align}
Hence,
\begin{align}\label{sof3}
  \biggabs{ \prod_{i\le k<j}(1+\gl/w_k)}&
=\exp\Bigpar{\frac{\Re\gl}{b}\bigpar{\log j -\log i}+O(1)}
\notag\\&
\le C \parfrac{j}{i}^{\Re\gl/b}.
\end{align}

We turn to the final (matrix) product in \eqref{sof1} and expand it into a
polynomial 
$\sum_\ell a_\ell N^\ell$ in $N$. 
The coefficient $a_\ell$ of $N^\ell$ is a sum 
of the product $\prod_{j=1}^\ell(w_{k_j}+\gl)\qw$
over $\ell$-tuples $k_1<\dots< k_\ell$ of indices.
Hence,
\begin{align}\label{sof4}
|  a_\ell|
\le \Bigpar{\sum_{i\le k<j}|w_k+\gl|\qw }^\ell
.\end{align}
We have, similarly to \eqref{sof2},
\begin{align}\label{sof5}
  \sum_{i\le k<j}|w_k+\gl|\qw 
=
  \sum_{i\le k<j}\Bigpar{\frac{1}{bk} + O\Bigpar{\frac{1}{k^2}}}
\le C\Bigpar{1+\log \frac ji}.
\end{align}
Since $N^\ell=0$ for $\ell>\nu_\gl$, we only have to consider $a_\ell
N^\ell$ for $\ell\le \nu_\gl$. Hence, \eqref{sof4}--\eqref{sof5} imply
\begin{align}\label{sof6}
\Bignorm{\prod_{i\le k<j}\bigpar{I+(w_k+\gl)\qw N}}
\le \sum_{\ell=0}^{\nu_\gl}C\Bigpar{1+\log \frac ji}^{\ell}
\le C\Bigpar{1+\log \frac ji}^{\nu_\gl}
.\end{align}
Combining \eqref{sof1}, \eqref{sof3} and \eqref{sof6} we obtain
\eqref{lsof}.

This proof of \eqref{lsof} assumed that $i\ge i_0$ for some $i_0$ such that
$w_{i_0}\ge 2|\gl|$. However, for $i<i_0$ we may use \eqref{lsof} with $i=i_0$
and multiply by the missing factors in \eqref{sof1a}
(which are bounded) to conclude that
\eqref{lsof} holds for all $1\le i\le j<\infty$.
\end{proof}

\begin{lemma}\label{Lsoff}
With notations as above, for every eigenvalue $\gl\in\gs(A)$
and $n\ge2$,
  \begin{align}
    \sum_{i=1}^n\norm{P_\gl F_{i,n}}^2
\le
    \begin{cases}
 C n,& \Re\gl <b/2,\\
 C n \log^{1+2\nu_\gl}n,& \Re\gl =b/2,\\
 C n^{2\Re \gl/b} \log^{2\nu_\gl}n,& \Re\gl >b/2,
    \end{cases}
  \end{align}
for some constant $C$ not depending on $n$.
\end{lemma}
\begin{proof}
  This follows from \refL{Lsof} by simple calculations.
Let $\gam:=\Re\gl/b$. 
If $\gam<\frac12$, we obtain from \refL{Lsof}, by comparing the sum to a
convergent integral,
  \begin{align}\label{lsoff1}
\sumin  \norm{P_\gl F_{i,n}}^2&
\le C \sumin \parfrac{i}{n}^{-2\gam}\Bigpar{1+\log \frac ni}^{2\nu_\gl}
\notag\\&
\le C n\intoi x^{-2\gam}\Bigpar{1+\log \frac 1x}^{2\nu_\gl}\dd x
=Cn
.\end{align}
If $\gam=\frac12$ we obtain instead
  \begin{align}\label{lsoff2}
\sumin  \norm{P_\gl F_{i,n}}^2&
\le C n\sumin {i}^{-1}\bigpar{1+\log n}^{2\nu_\gl}
\le C n(\log n)^{1+2\nu_\gl}
.\end{align}
Finally, if $\gam>\frac12$, we obtain
  \begin{align}\label{lsoff3}
\sumin  \norm{P_\gl F_{i,n}}^2&
\le C \sumin \parfrac{n}{i}^{2\gam}\bigpar{\log n}^{2\nu_\gl}
\le C n^{2\gam}\bigpar{\log n}^{2\nu_\gl}\sumi i^{-2\gam}
,\end{align}
where the final sum is convergent.
\end{proof}

\appendix

\section{Uniform integrability}\label{Aui}

The purpose of this appendix is to prove the following version of \refT{T3},
useful in the (unusual) case of a balanced urn 
where  the increments $\xi_i$ have some but not
all moments finite.

\begin{theorem}
  Assume that the conditions of \refT{T3}\ref{T3a} or \ref{T3b} holds,
except that \eqref{t1p} is assumed only for a single $p\ge2$.
Then \eqref{t3a} or \eqref{t3b} holds with convergence of all moments of
order less than or equal to $p$.
\end{theorem}

\begin{proof}
  We argue as in the proof of \refT{T1} above, but use \refL{LLui} below
  instead of \refL{LL2}. 
It follows that if $L_n$ is the random vector on the \lhs{} of \eqref{t1}
and $r_n$ is the function of $n$ on the \rhs, then 
$|L_n/r_n|^p$ forms a \ui{} sequence. It follows as in the
proof of \refT{T2} that the
same holds for \eqref{t2}; in other words, the \lhs{} of \eqref{t3a} or
\eqref{t3b} has \ui{} $p$th powers. Hence convergence of $p$th moments follows
from the assumed convergence in distribution.
\end{proof}

The proof above uses the following lemma, which was shown to me by an
anonymous referee of \cite{SJ307}.
We do not know any reference, so we give a complete proof.

\begin{lemma}\label{LLui}
Let $p\ge2$ and let\/ $Y_i$, $i\ge1$, be a \mgds{} 
of random vectors in $\bbC^q$
such that
the variables $|Y_i|^p$, $i\ge1$, are uniformly integrable.
Then the collection of random variables
\begin{align}\label{llui}
\lrset{\biggabs{\sumi A_i Y_i}^p:A_i\in\bbC^{q\times q},\, \sumi\norm{A_i}^2\le1}
\end{align}
  is \ui. (The first sum in \eqref{llui} converges a.s.)
\end{lemma}

Note that the a.s.\ convergence of $\sum_i A_iY_i$ follows from \refL{LL2},
which also shows that the $L^p$-norms of these sums are bounded.
We obtain Lemma \ref{LLui} from \refL{LL2} by a simple truncation argument.

\begin{proof}[Proof of \refL{LLui}]
Let $\eps>0$. By assumption, there exists $M=M(\eps)$ such that
\begin{align}
  \label{lui1}
\E\bigabs{Y_i\indic{|Y_i|> M}}^p <\eps^p,
\qquad i\ge1.
\end{align}
Let
\begin{align}
 Y_i'&:= Y_i\indic{|Y_i|\le M} - \E \bigpar{Y_i\indic{|Y_i|\le M}\mid \cF_{i-1}},
\\
 Y_i''&:= Y_i\indic{|Y_i|> M} - \E \bigpar{Y_i\indic{|Y_i|> M}\mid \cF_{i-1}}.
\end{align}
Then $Y_i=Y_i'+Y_i''$. Furthermore,
\begin{align}
  |Y_i'|&\le 2M \qquad \text{a.s.}, \label{luiyi'}
\\
\norm{Y_i''}_p &\le 2\bignorm{Y_i\indic{|Y_i|> M}}_p < 2\eps. \label{luiyi''}
\end{align}
Let $r:=p+1$ (any fixed $r>p$ will do), and note that \eqref{luiyi'} implies
\begin{align}\label{luiq'}
  \norm{Y_i'}_r\le 2M.
\end{align}

Let $(A_i)\xoo$ be any sequence of matrices with $\sum_i\norm{A_i}^2\le1$. Then
\refL{LL2} yields, together with
\eqref{luiq'} and \eqref{luiyi''},
\begin{align}
  \biggnorm{\sum_i A_i Y_i'}_r&\le 2C_r M, \label{lulu'}
\\
  \biggnorm{\sum_i A_i Y_i''}_p&\le 2C_p\eps,\label{lulu''}
\end{align}

Let $\gd>0$ and let $\cE$ be any event with $\P(\cE)\le\gd$.
We have $\sum_i A_i Y_i = \sum_i A_i Y_i'+\sum_i A_i Y_i''$, and thus (crudely),
\begin{align}
  \biggabs{\sum_i A_i Y_i}^p 
\le 2^p\biggabs{\sum_i A_i Y_i'}^p+2^p\biggabs{\sum_i A_i Y_i''}^p.
\end{align}
Hence, using \Holder's inequality with $s:=(r/p)'=r/(r-p)$,
and \eqref{lulu'}--\eqref{lulu''},
\begin{align}
\E\lrpar{\biggabs{\sum_i A_i Y_i}^p \indic{\cE}}  &
\le
2^p\E\lrpar{\biggabs{\sum_i A_i Y_i'}^p \indic{\cE}}  
+2^p \E\lrpar{\biggabs{\sum_i A_i Y''_i}^p \indic{\cE}} 
\notag\\&
\le
2^p\biggnorm{\Bigabs{\sum_i A_i Y_i'}^p}_{r/p}\bignorm{ \indic{\cE}}_s  
+2^p \E\lrpar{\biggabs{\sum_i A_i Y''_i}^p } 
\notag\\&
=
2^p\biggnorm{\sum_i A_i Y_i'}_{r}^p \P\xpar{\cE}^{1/s}  
+2^p \biggnorm{\sum_i A_i Y''_i}_p^p 
\notag\\&
\le 2^p (2C_r M)^p \gd^{1/s} + 2^p(2C_p\eps)^p.
\label{luiq}
\end{align}
For any $\eta>0$, we can make the \rhs{} of \eqref{luiq}
$<\eta+\eta$ by first choosing $\eps$ and then $\gd$ small enough.
Consequently,
\begin{align}\label{llx}
  \lim_{\gd\to0}\sup_{\P(\cE)\le\gd,\,\sum_i \norm{A_i}^2\le1}
\E\lrpar{\biggabs{\sum_i A_i Y_i}^p \indic{\cE}} =0.
\end{align}
Finally, the assumption implies $\sup_i\E|Y_i|^p<\infty$, and thus
another application of \refL{LL2} yields
$\sup_{\sum_i \norm{A_i}^2\le1}\E\bigabs{\sum_i A_i Y_i}^p  <\infty$,
which together with \eqref{llx} shows the uniform integrability of \eqref{llui}.
\end{proof}

\section{Proof of \refL{Lsimple}}\label{AP}
\begin{proof}
  Consider an eigenvalue  $\gl\neq\gl_1=b$. Then \eqref{24b} shows that
$(A-\gl I)P_\gl=N_\gl$, and thus, since $\nu_\gl<q$, 
$(A-\gl I)^qP_\gl=((A-\gl I)P_\gl)^q=0$. (Recall that $P_\gl$ is a
projection and that it commutes with $A$.)
Hence, for any vector $v\in\bbC^q$,  since $u_1(A-\gl I)=(\gl_1-\gl)u_1$,
\begin{align}
  0=u_1'(A-\gl I)^qP_\gl v
=(\gl_1-\gl)^qu_1P_\gl v,
\end{align}
and consequently
\begin{align}
  u_1'P_\gl v=0.
\end{align}
This says that $u_1$ is orthogonal to the generalized eigenspace 
$E_\gl=P_\gl\bbC^q$
for every $\gl\neq\gl_1$. 
Since $\bbC^q=\bigoplus_\gl E_\gl$ and $u_1\neq0$, it follows that $u_1$ is not
orthogonal to $E_{\gl_1}=\set{zv_1:z\in\bbC}$. Hence $u_1\cdot v_1\neq0$ and
we may choose $v_1$ such that \eqref{normalized} holds.

Since $P_{\gl_1}$ is a projection onto the eigenspace $E_{\gl_1}$ spanned by
$v_1$ we have
\begin{align}\label{sjw2}
P_{\gl_1}v=(u\cdot v)v_1  ,
\qquad v\in\bbC^q
\end{align}
for some vector
$u$ with $u\cdot v_1=1$. Furthermore, 
\begin{align}
u_1'P_{\gl_1}A  
=
u_1'AP_{\gl_1}
=\gl_1 u_1'P_{\gl_1}
\end{align}
and thus $u_1'P_{\gl_1}$ is a left eigenvector of $A$ with eigenvalue
$\gl_1$
and thus a multiple of $u_1$. Since $P_{\gl_1}$ is a projection, it follows that
$u_1'P_{\gl_1}=u_1'$. Hence \eqref{sjw2} implies that, for any $v\in\bbC^q$,
\begin{align}
  u_1'v =u_1'P_{\gl_1}v =(u\cdot v)(u_1\cdot v_1)
=u\cdot v.
\end{align}
Consequently, $u=u_1$; thus \eqref{sjw2} yields \eqref{pl2}, and thus also
\eqref{pl1}.
\end{proof}

\newcommand\AAP{\emph{Adv. Appl. Probab.} }
\newcommand\JAP{\emph{J. Appl. Probab.} }
\newcommand\JAMS{\emph{J. \AMS} }
\newcommand\MAMS{\emph{Memoirs \AMS} }
\newcommand\PAMS{\emph{Proc. \AMS} }
\newcommand\TAMS{\emph{Trans. \AMS} }
\newcommand\AnnMS{\emph{Ann. Math. Statist.} }
\newcommand\AnnPr{\emph{Ann. Probab.} }
\newcommand\CPC{\emph{Combin. Probab. Comput.} }
\newcommand\JMAA{\emph{J. Math. Anal. Appl.} }
\newcommand\RSA{\emph{Random Struct. Alg.} }
\newcommand\ZW{\emph{Z. Wahrsch. Verw. Gebiete} }
\newcommand\DMTCS{\jour{Discr. Math. Theor. Comput. Sci.} }

\newcommand\AMS{Amer. Math. Soc.}
\newcommand\Springer{Springer-Verlag}
\newcommand\Wiley{Wiley}

\newcommand\vol{\textbf}
\newcommand\jour{\emph}
\newcommand\book{\emph}
\newcommand\inbook{\emph}
\def\no#1#2,{\unskip#2, no. #1,} 
\newcommand\toappear{\unskip, to appear}

\newcommand\arxiv[1]{\texttt{arXiv:#1}}
\newcommand\arXiv{\arxiv}

\def\nobibitem#1\par{}

\end{document}